\documentclass[12pt,aps,tightenlines,notitlepage,superscriptaddress,showkeys]{revtex4-1}

\usepackage[margin=3cm]{geometry}

\usepackage{amsmath}
\usepackage{amssymb}
\usepackage{amsfonts}
\usepackage{bm}
\usepackage[mathscr]{eucal}
\usepackage{theorem}
\usepackage{graphicx}
\usepackage{psfrag}
\usepackage{subfigure}
\usepackage{color}
\usepackage{wasysym}
\include{graphix}
\usepackage{framed}
\usepackage{enumitem}%
\usepackage{lipsum}
\usepackage{float}
\usepackage{MnSymbol}
\usepackage{url}

\graphicspath{{figs/}}

\newcommand{\imag}{\mathrm{i}}
\newcommand{\mathe}{\mathrm{e}}
\newcommand{\mathd}{\mathrm{d}}
\newcommand{\mydop}{\mathcal{D}}

\newtheorem{remark}{Remark}
\newtheorem{theorem}{Theorem}
\newenvironment{proof}[1][Proof]{\begin{trivlist}
\item[\hskip \labelsep {\bfseries #1}]}{\end{trivlist}}

\newcommand{\myqed}{\nobreak \ifvmode \relax \else
      \ifdim\lastskip<1.5em \hskip-\lastskip
      \hskip1.5em plus0em minus0.5em \fi \nobreak
      \vrule height0.75em width0.5em depth0.25em\fi}

\begin{document}
\title{A mathematical framework for determining the stability of steady states of reaction-diffusion equations with periodic source terms}
\author{Lennon \'O N\'araigh\footnote{Email address: \texttt{onaraigh@maths.ucd.ie}} and Khang Ee Pang}
\address{School of Mathematics and Statistics, University College Dublin, Belfield, Dublin 4, Ireland}


\date{\today}

\begin{abstract}
We develop a mathematical framework for determining the stability of steady states of generic nonlinear reaction-diffusion equations with periodic source terms, in one spatial dimension.  We formulate an \textit{a priori} condition for the stability of such steady states, which relies only on the properties of the steady state itself.   The mathematical framework is based on Bloch's theorem and Poincar\'e's inequality for mean-zero periodic functions.  Our framework can be used for stability analysis to determine the regions in an appropriate parameter space for which steady-state  solutions are stable. 
\end{abstract}


\maketitle

\section{Introduction and Problem Statement}
\label{sec:intro}

\noindent 

We introduce a mathematical framework for analyzing the stability of equilibrium solutions of generic reaction-diffusion equations with a periodic source term.  Nonlinear reaction-diffusion equations occur in the context of pattern formation, chemical reactions, mathematical biology, and phase separation of 
binary alloys~\cite{grindrod1996theory,murray2011mathematical,allen1972ground}).  The addition of a forcing term in such systems of equations can be used effectively to drive a chemical reaction to a desired outcome~\cite{lin2004resonance}.  Equally, a source term of travelling-wave type can be used to control the naturally-occurring travelling waves in such systems~\cite{bonhours2015}.
Motivated by these applications, in this article we consider the generic nonlinear reaction-diffusion equation
\begin{equation}
\frac{\partial C}{\partial t}=\sigma(x)+N(C)+\frac{\partial^2C}{\partial x^2},\qquad x\in (-\infty,\infty),\qquad t>0,
\label{eq:model_C}
\end{equation}
where $N(C)$ is a smooth nonlinear function of the variable $C$, and $\sigma(x)$ is a periodic source term, with $\sigma(x+L)=\sigma(x)$.  We seek periodic steady-state solutions $C_0(x)$ that satisfy
\begin{equation}
0=\sigma(x)+N(C_0)+\frac{\partial^2C_0}{\partial x^2},\qquad x\in (0,L),\qquad C_0(x+L)=C_0(x).
\label{eq:model_C_ss}
\end{equation}
If such solutions can be found, it is of interest to classify their stability according to linear stability theory.  As such, in this article we consider a solution
\begin{equation}
C(x,t)=C_0(x)+\delta C(x,t),
\label{eq:model_C_perturb_val}
\end{equation}
where $\delta C(x,t)$ is a small perturbation.  By substituting Equation~\eqref{eq:model_C_perturb_val} into Equation~\eqref{eq:model_C} and linearizing the nonlinear term $N(C)$, a linearized partial differential equation for the perturbation $\delta C(x,t)$ is obtained, which is valid provided the magnitude of the perturbation remains small in the sense that $|\delta C|\ll |N''(C_0)/N'(C_0)|$, for all $x$ and $t>0$ (here, $N'$ denotes the derivative of $N$ with respect to its argument and $N''$ denotes the second derivative).  As such, the following linearized partial differential equation for $\delta C$ is obtained:
\begin{equation}
\frac{\partial }{\partial t}\delta C=N'(C_0)\delta C+\frac{\partial^2}{\partial x^2}\delta C,\qquad x\in (-\infty,\infty),\qquad t>0.
\label{eq:model_C_perturb}
\end{equation}
The boundary conditions for Equation~\eqref{eq:model_C_perturb} are provided based on physical intuition, namely that the perturbations should vanish as $|x|\rightarrow \infty$.  In this article, we introduce a mathematical framework that enables us to derive sufficient conditions such that the solution $C_0(x)$ is stable, i.e. such that $\lim_{t\rightarrow \infty}|\delta C(x,t)|=0$.
As such, we emphasize that the purpose of this article is not to construct base-state solutions $C_0(x)$ corresponding to Equation~\eqref{eq:model_C_ss}, but rather to determine \textit{a priori} the stability of such solutions, once $C_0(x)$ is known.  General results concerning the existence of periodic solutions of Equation~\eqref{eq:model_C_ss} can be found elsewhere in the 
literature~\cite{seifert1959note,knobloch1962existence}.

This article is organized as follows.  In Section~\ref{sec:eigenvalue} we reduce Equation~\eqref{eq:model_C_perturb} to an eigenvalue problem, and we characterize the eigenvalues and eigenfunctions.  We formulate the condition for linear stability of the solution $C_0(x)$ in terms of eigenvalues of a certain linear operator.  In Section~\ref{sec:sufficient} we formulate sufficient conditions  such that $C_0(x)$ is a stable equilibrium solution of Equation~\eqref{eq:model_C}.  In Section~\ref{sec:disc} we compare our computed conditions with prior work in the literature.  We also apply our conditions in one specific case which is relevant to applications, and where an exact strability criterion in terms of a series expansion is known (Matthieu's equation).
Concluding remarks are presented in Section~\ref{sec:conc}.

\section{Eigenvalue  analysis}
\label{sec:eigenvalue}

We consider Equation~\eqref{eq:model_C_perturb}, written here in abstract terms as follows:
\begin{subequations}
\begin{equation}
\frac{\partial \phi}{\partial t}=s(x)\phi+\frac{\partial^2\phi}{\partial x^2},\qquad x\in (-\infty,\infty),\qquad t>0,
\label{eq:model}
\end{equation}
with smooth square-integrable initial condition
\begin{equation}
\phi(x,t=0)=\phi_0(x),\qquad \phi_0\in L^2(-\infty,\infty),
\label{eq:model_ic}
\end{equation}
and where $s(x)$ is an $L$-periodic function, with $s(x+L)=s(x)$.  We further specify the behavioral boundary condition
\begin{equation}
\phi(x)\rightarrow 0 \text{ as }|x|\rightarrow \infty.
\label{eq:model_bc}
\end{equation}%
\label{eq:model_all}%
\end{subequations}%
We apply separation of variables to Equation~\eqref{eq:model_all}, with $\phi(x,t)=\mathe^{-\lambda t}\psi(x)$.  This yields
\begin{equation}
-\lambda \psi(x)=\left[s(x)+\frac{\partial^2}{\partial x^2}\right]\psi(x),
\label{eq:eigpde}
\end{equation}
The linear operator $\mathcal{L}=s(x)+\partial_{xx}$ is translation-invariant under the group operation $x\rightarrow x+L$.  The corresponding group action on functions can be written in terms of a translation operator $\mathcal{T}_Lf(x)=f(x+L)$, for a generic function $f(x)$.  Thus, $\mathcal{L}$ and $\mathcal{T}_L$ commute, and are therefore simultaneously diagonalizable.  As such, the eigenfunctions of $\mathcal{L}$ can be written as 
\begin{equation}
\psi(x)=\psi_{np}(x)=\mathe^{\imag px}u_{np}(x),\qquad p\in [-\kappa/2,\kappa/2],\qquad n\in \{0,1,\cdots\},
\qquad \kappa=2\pi/L,
\label{eq:eigsln}
\end{equation}
where $u_{np}(x)$ is an $L$-periodic function and solves the self-adjoint problem
\begin{equation}
-\lambda_{np}u_{np}=\left[s(x)-p^2+2\imag p\frac{\partial}{\partial x}+\frac{\partial^2}{\partial x^2}\right]
u_{np}
\label{eq:bloch_op}
\end{equation}%
\begin{remark}
Equation~\eqref{eq:eigsln} is a particular application of Bloch's Theorem.
\end{remark}
By self-adjointness of the operator in Equation~\eqref{eq:bloch_op} the eigenfunctions $u_{np}$ can be normalized to satisfy
\begin{equation}
\int_0^L u_{np}^*(x)u_{n'p}(x)\,\mathd x=(L/2\pi)\delta_{nn'},
\end{equation}%
where the star indicates complex conjugation.
Hence also,
\begin{eqnarray*}
\int_{-\infty}^\infty \psi_{np}^*(x)\psi_{n'p'}(x)\,\mathd x&=&
\sum_{j=-\infty}^\infty \int_{jL}^{(j+1)L} \mathe^{\imag (p-p')x} u_{np}^*(x)u_{n'p'}(x)\,\mathd x,\\
&=&\sum_{j=-\infty}^\infty \mathe^{\imag (jL)(p-p')} \int_{0}^{L} \mathe^{\imag (p-p')x} u_{np}^*(x)u_{n'p'}(x)\,\mathd x,\\
&=&(2\pi/L) \delta(p-p') \int_{0}^{L} \mathe^{\imag (p-p')x} u_{np}^*(x)u_{n'p'}(x)\,\mathd x,\\
&=&\delta(p-p')\delta_{nn'},
\end{eqnarray*}
which is a completeness relation for the Bloch eigenfunctions.  As such,
\[
\phi_0(x)=\int_{-\kappa/2}^{\kappa/2} \left[\sum_{n=0}^\infty \langle \psi_{np},\phi_0\rangle \psi_{np}(x)\right]\mathd p,
\]
where
\[
\langle \psi_{np},\phi_0\rangle=\int_{-\infty}^\infty \left[\mathe^{\imag px} u_{np}(x)\right]^* \phi_0(x)\,\mathd x
\]
denotes the  inner product of $\psi_{np}$ with $\psi_0$ (the integral makes sense because $\phi_0(x)\in L^2(-\infty,\infty)$).  Hence, the general solution of Equation~\eqref{eq:model_all} reads
\[
\phi(x,t)=
\int_{-\kappa/2}^{\kappa/2} \left[\sum_{n=0}^\infty \langle \psi_{np},\phi_0\rangle \psi_{np}(x)\mathe^{-\lambda_{np}t}\right]\mathd p.
\]
Hence, in order for the solution $\phi(x,t)$ to remain bounded for all time, we require $\lambda_{np}\geq 0$, for all $p\in[-\kappa/2,\kappa/2]$ and all $n\in\{0,1,\cdots\}$, in Equation~\eqref{eq:bloch_op}.  This is precisely the condition for the steady-state solution of Equation~\eqref{eq:model_C} in Section~\ref{sec:intro} to be linearly stable.

\section{Sufficient conditions for stability}
\label{sec:sufficient}

In this section, we use  known results in the literature on the spectral theory for linear operators with periodic boundary conditions and determine criteria for $s(x)$ such that $\lambda_{np}\geq 0$, for all $p\in[-\kappa/2,\kappa/2]$ and all $n\in\{0,1,\cdots\}$, in Equation~\eqref{eq:bloch_op}.  For these purposes, it suffices to notice that Equation~\eqref{eq:eigpde} is in the form of a one-dimensional Schr\"odinger equation with periodic potential $-s(x)$.  Due to the importance of the Schr\"odinger equation with periodic potential
in applications~\cite{kittel1976introduction}, the spectral properties of this equation are well 
studied~\cite{eastham1973spectral,kobayashi1989periodic,ren2017electronic}; in particular it has been shown rigorously that the minimum eigenvalue occurs at $p=0$ and $n=0$~\cite{ren2017electronic}, such that
\[
\lambda_{np}\geq \lambda_{00},\qquad p \in [-\kappa/2,\kappa/2],\qquad n\in\{0,1,\cdots\}.
\]
Therefore, in order to establish the positivity of $\lambda_{np}$, it suffices to establish the positivity of $\lambda_{00}$.

Motivated by this discussion, we return to the eigenvalue analysis of Equation~\eqref{eq:bloch_op} with both $p=0$ and $n=0$, to determine sufficient conditions for $\lambda_{00}\geq 0$.  As such, 
we multiply Equation~\eqref{eq:bloch_op} (with $p=0$ and $n=0$) by $u_{00}^*$ and integrate from $x=0$ to $x=L$, applying the periodic boundary conditions to $u_{00}(x)$.  We obtain (after suppressing the subscripts $p=0$ and $n=0$):
\begin{equation}
\lambda \|u\|_2^2 = -\langle u,s u \rangle+\|\mathd u/\mathd x\|_2^2,
\end{equation}
where in this context the angle brackets $\langle \cdot,\cdot\rangle$ denote the usual $L^2$ inner product,
\[
\langle f,g\rangle=\int_0^L (f^*)g\,\mathd x,\qquad f,g,\in L^2([0,L]).
\]
Similarly, $\|f\|_2^2=\langle f,f\rangle$, for all $f\in L^2([0,L])$.
Clearly, if $s(x)\leq 0$ for all $x\in [0,L]$, then $\lambda\geq 0$.  This is certainly a sufficient condition for stability, however, it is highly prescriptive.  As such, introduce the mean value of a function,
\[
\langle f\rangle=\frac{1}{L}\int_0^L f(x)\mathd x,
\]
for any integrable $f$ function on $[0,L]$.  Accordingly, we examine $\langle s \rangle$, and we look at the possibility that $\langle s\rangle<0$ but that $\max_{[0,L]}s(x)=s_0\geq 0$.  We therefore derive constraints on $\langle s\rangle$ and $s_0$ such that $\lambda\geq 0$.  We summarize the results here:
\begin{theorem}
\label{thm:smax}
Let 
\[
\langle s\rangle=L^{-1}\int_0^L s(x)\,\mathd x,\qquad \delta s=s-\langle s\rangle.
\]
Suppose that $\langle s\rangle \leq 0$.  
The minimum eigenvalue $\lambda$ ($\lambda_{00}$ with subscripts restored) of Equation~\eqref{eq:bloch_op} is positive if
\begin{itemize}
\item $s(x)\leq 0$ for all $x\in [0,L]$, or less restrictively,
\item If $s_0=\max_{[0,L]}s(x)$ is positive, but 
%
%
\begin{subequations}
\begin{eqnarray}
s_0&\leq& \kappa^2,\label{eq:smaxa}\\
\|\delta s\|_2^2&\leq& \left(\kappa^2-s_0\right)|\langle s\rangle|.\label{eq:smaxb}
\end{eqnarray}
\label{eq:smax}
\end{subequations}

\end{itemize}
\end{theorem}
	
\begin{proof}
The starting-point for the proof is Equation~\eqref{eq:bloch_op} (with $p=0$ and $n=0$) and its averaged version, written here with subscripts suppressed as
\begin{eqnarray*}
-\lambda u  &=& su +\mydop u ,\\
-\lambda \langle u \rangle&=&\langle s u \rangle,
\end{eqnarray*}
where $\mydop=\mathd^2/\mathd x^2$; here we have used the periodic boundary conditions on $u$ to write $\langle \mydop u\rangle=0$.  

We would like to apply Poincar\'e's inequality for mean-zero functions on the periodic domain $[0,L]$ to obtain an estimate on the eigenvalue $\lambda$, however, we are precluded from doing so directly, as it cannot be assumed that $\langle u\rangle=0$.  We therefore proceed by further decomposing $s$ and $u$ in terms of mean components and fluctuations:
\[
s=\langle s\rangle +\delta s,\qquad u=\langle u\rangle+\delta u.
\]
Hence,
\[
-\lambda \langle u \rangle=\langle s\rangle\langle u \rangle+\langle \delta s \delta u\rangle.
\]
We identify two cases:
\begin{itemize}
\item If $\lambda+\langle s\rangle=0$, then $\lambda=-\langle s\rangle\geq 0$ and the eigenvalue $\lambda$ is non-negative.
\item Otherwise, $\lambda+\langle s\rangle \neq 0$.
\end{itemize}
If Case 1 pertains, the eigenvalue $\lambda$ is positive, and we are done.  We therefore assume that Case 2 pertains, hence
\begin{equation}
\langle u\rangle=-\frac{ \langle \delta s \delta u\rangle}{\lambda+\langle s\rangle}.
\label{eq:meanup}
\end{equation}
We now further rewrite the eigenvalue problem~\eqref{eq:bloch_op} in terms of $\delta u$ and $\delta s$:
\[
-\lambda \delta u = \langle s \rangle \delta u+\langle u\rangle \delta s+\delta s \delta u-\langle \delta s\delta u\rangle+\mydop\delta u.
\]
We multiply both sides by $\delta u^*$ and integrate from $x=0$ to $x=L$.  We obtain
\begin{equation}
-\lambda \|\delta u\|_2^2=\langle s\rangle \|\delta u\|_2^2+\langle u\rangle\langle \delta u,\delta s\rangle + \langle \delta u,\delta s\delta u\rangle+\langle \delta u,\mydop\delta u\rangle.
\label{eq:deltaup}
\end{equation}
This can furthermore be written as
\begin{equation}
-\lambda \|\delta u\|_2^2= 
 \langle \delta u, s\delta u\rangle+
\langle u\rangle\langle \delta u,\delta s\rangle +\langle \delta u,\mydop\delta u\rangle.
\label{eq:deltaup1}
\end{equation}
We use  Equation~\eqref{eq:meanup} to eliminate $\langle u\rangle$ from Equation~\eqref{eq:deltaup}:
\[
%
\lambda \|\delta u\|_2^2=-\langle \delta u, s\delta u\rangle-\langle \delta u,\mydop\delta u\rangle+\frac{ \langle \delta s \delta u\rangle}{\lambda+\langle s\rangle}\langle \delta u,\delta s\rangle.
\]
Since $s(x)$ is real-valued, we have
\[
\langle \delta s \delta u \rangle \langle \delta u,\delta s\rangle=|\langle \delta u,\delta s\rangle|^2,
\]
hence
\begin{equation}
%
\lambda \|\delta u\|_2^2=\underbrace{-\langle \delta u, s\delta u\rangle-\langle \delta u,\mydop\delta u\rangle}_{=Q}+\frac{ |\langle \delta s, \delta u\rangle|^2}{\lambda+\langle s\rangle}.
\label{eq:abc1}
\end{equation}
In what follows, it will be necessary to have $Q\geq 0$.  We now formulate conditions on $s(x)$ such that $Q\geq 0$.  We have
\begin{eqnarray*}
Q&=& -\langle \delta u, s\delta u\rangle-\langle \delta u,\mydop\delta u\rangle,\\
 &=& -\langle \delta u, s\delta u\rangle+\|\frac{\mathd}{\mathd x}\delta u\|_2^2,\\
 &\geq & -s_0  \|\delta u\|_2^2+\|\frac{\mathd}{\mathd x}\delta u\|_2^2.
\end{eqnarray*}
We notice that $\delta u$ is a mean-zero differentiable function on the periodic domain $[0,L]$.  We therefore use
 Poincar\'e's inequality  $\|(\mathd/\mathd x)\delta u\|_2^2\geq \kappa^2 \|\delta u\|_2^2$ to write
\[
Q \geq \left(-s_0+\kappa^2\right)\|\delta u\|_2^2.
\]
Hence, we require
\[
s_0\leq  \kappa^2,
\]
which is exactly Equation~\eqref{eq:smaxa}

Continuing, Equation~\eqref{eq:abc1} can be written as
\begin{equation}
\lambda=a+\frac{b}{\lambda-c},
\label{eq:abc2}
\end{equation}
where $a$, $b$, and $c$ are positive real numbers:
\begin{equation*}
%
a=-\frac{\langle \delta u, s\delta u\rangle+\langle \delta u,\mydop\delta u\rangle}{\|\delta u\|_2^2}
=
 -\frac{\langle \delta u, s\delta u\rangle}{\|\delta u\|_2^2}
+\frac{\|\frac{\mathd}{\mathd x}\delta u\|_2^2}{\|\delta u\|_2^2},
\end{equation*}
and
\begin{equation*}
b= \frac{|\langle \delta u,\delta s\rangle|^2}{\|\delta u\|_2^2},\qquad
c=|\langle s\rangle|.
\end{equation*}
Equation~\eqref{eq:abc2} has solutions
\begin{equation}
\lambda=\frac{a+c}{2}\left[1\pm \sqrt{1-\frac{4(ac-b)}{(a+c)^2}}\right].
\label{eq:abc3}
\end{equation}
Since Equation~\eqref{eq:bloch_op} is a self-adjoint problem, the eigenvalues $\lambda$ are real-valued.  Furthermore, to make $\lambda\geq 0$ for both signs in Equation~\eqref{eq:abc3}, it suffices to make $b\leq ac$, hence
\[
%
\frac{|\langle \delta u,\delta s\rangle|^2}{\|\delta u\|_2^2}\leq 
\left[ -\frac{\langle \delta u, s\delta u\rangle}{\|\delta u\|_2^2}+
\frac{\|\frac{\mathd}{\mathd x}\delta u\|_2^2}{\|\delta u\|_2^2}\right]|\langle s\rangle|.
\]
A sufficient condition for this to be true for $\delta u$ is if
\[
%
\sup_{\phi\neq 0} \frac{|\langle \phi,\delta s\rangle|^2}{\|\phi\|_2^2}\leq 
\inf_{\phi\neq 0}\left[-s_0+\frac{\|\frac{\mathd\phi}{\mathd x}\|_2^2}{\|\phi\|_2^2}\right]|\langle s\rangle|,
\]
where the sup and inf are taken over all mean-zero differentiable functions on $[0,L]$.  The sup and inf can be calculated using the Cauchy--Schwarz and Poincar\'e inequalities respectively: 
%
%
\[
\|\delta s\|_2^2\leq \left(-s_0+\kappa^2\right)|\langle s\rangle|.
\]
This condition is precisely Equation~\eqref{eq:smaxb}.  This concludes the proof. 
\end{proof}

\section{Discussion}
\label{sec:disc}

In this section we compare the computed bound~\eqref{eq:smax} with prior work in the literature.  We also apply the bound in one specific case which is relevant to applications, and where an exact series solution for the stability boundary is known (Matthieu's equation).

\subsection{Comparison with other works}

 Equation~\eqref{eq:bloch_op} with $p=0$ amounts to a standard Sturm--Liouville eigenvalue problem with periodic boundary conditions.  There is much prior work in the literature concerning upper and lower bounds for the eigenvalues of Sturm--Liouville systems.  However, much of this literature is for separated eigenvalue problems (for instance, Reference~\cite{breuer1971upper} have computed a lower bound for the eigenvalue in the case of Dirichlet boundary conditions using the Sturm--Picone comparison theorem).  However, it is not immediately obvious that such work carries over to the periodic case.
Instead, we look at a more specific literature, since the eigenvalue problem $[s(x)+\partial_{xx}]u=-\lambda u$ (with periodic boundary conditions) is precisely Hill's equation, which is also well studied.
In particular, the least eigenvalue $\lambda_{00}$ is known to satisfy~\cite{kato1952note}
\[
\lambda_{00}\geq -\langle s\rangle-\frac{1}{8L} \|\delta s\|_2^2.
\]
Therefore, with $\langle s\rangle\leq 0$, $\lambda_{00}$ is positive provided
\begin{equation}
\|\delta s\|_2^2\leq 8 L |\langle s\rangle|.
\label{eq:sharp}
\end{equation}
Equation~\eqref{eq:sharp} is an additional criterion for stability, i.e.
a further condition on $s(x)$ such that $\phi(x,t)$ in Equation~\eqref{eq:model_all} remains bounded for all time.  
Equation~\eqref{eq:sharp} is useful as it highlights the importance of working with $\langle s\rangle\leq 0$: a negative mean value of $s(x)$  term gives a positive contribution to the eigenvalue $\lambda_{00}$, this then counteracts the negative contribution by $\|\delta s\|_2^2$ such that an overall positive eigenvalue $\lambda_{00}$ can be achieved.  This fact underlies our previous assumption in Section~\ref{sec:sufficient} wherein we worked with $\langle s\rangle \leq 0$ throughout.

\subsection{Specific application -- Mathieu's equation}

To illustrate how the conditions in Theorem~\ref{thm:smax} may apply, we consider the following specific form for $s(x)$:
\begin{equation}
s=-\alpha+\beta\cos(2\kappa x),
\label{eq:smatthieu}
\end{equation}
where $\alpha$ and $\beta$ are positive constants.  Equation~\eqref{eq:smatthieu} corresponds to some notional base state $C_0(x)$ in Equation~\eqref{eq:model_C_ss}.  For example, for a certain choice of source term $\sigma(x$), Equation~\eqref{eq:smatthieu} corresponds to the linearization of the inhomogeneous Allen--Cahn reaction--diffusion equation.  This is illustrated in Appendix~\ref{sec:allencahn}.  Equally, Equation~\eqref{eq:smatthieu} 
 corresponds to Mathieu's differential equation with eigenvalue $\lambda$, i.e. $y''+[-\alpha + \beta \cos(2\kappa x)]y=-\lambda y$.  Using the computed bound in Equation~\eqref{eq:smax}, the following sufficient condition for stability can be derived for the present specific form of $s(x)$:
\begin{equation}
\beta=-\alpha+\sqrt{3\alpha^2+2\kappa^2\alpha}
\label{eq:bound3}
\end{equation}

Looking at the Mathieu equation is useful not only for the applications, but also because it is a standard equation for which an exact criterion for stability can be extracted from known results in the literature, in terms of a Taylor series in $\beta$.  Specifically, the exact criterion for stability for the functional form $s(x)=-\alpha+\beta\cos(2\kappa x)$ can be worked out as follows~\cite{abramowitz1965handbook}:
\begin{equation}
\frac{\alpha}{\kappa^2}=\tfrac{1}{2}(\beta/\kappa^2)^2-\tfrac{7}{128}(\beta/\kappa^2)^4+\tfrac{29}{2304}(\beta/\kappa^2)^6-\cdots.
\label{eq:exact}
\end{equation}
The results of this comparison are shown in Figure~\ref{fig:bound3_matthieu1x}.   Our own bound~\eqref{eq:bound3} is very close to the exact stability boundary for the Mathieu equation.  It is also much  sharper than the one given by Equation~\eqref{eq:sharp}.  In other words, our own bound predicts stability for a wider range of values of $\alpha$ and $\beta$ when compared with Equation~\eqref{eq:sharp}, and is therefore less restrictive.

\begin{figure}[!h]
	\centering
		\includegraphics[width=0.8\textwidth]{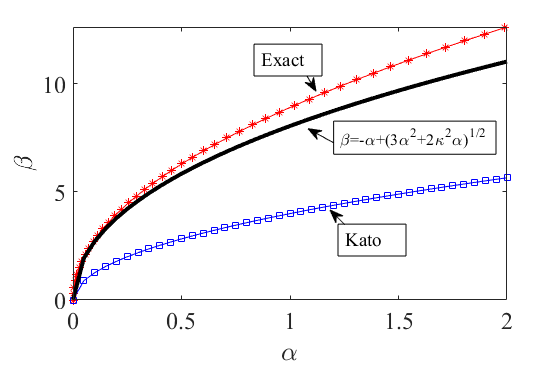}
		\caption{Region of stability for the Mathieu equation, i.e. the functional form $s(x)=-\alpha+\beta\cos(2\kappa x)$.  The exact stability boundary is computed using Equation~\eqref{eq:exact}.  This is compared with our own estimated stability criterion~\eqref{eq:bound3}, as well as a prior estimate from the literature~\cite{kato1952note}.}
	\label{fig:bound3_matthieu1x}
\end{figure}

\section{Conclusions}
\label{sec:conc}

Summarizing, we have formulated a model reaction-diffusion equation, posed on the entire real line, in the presence of a periodic source term, denoted by $\sigma(x)$.  We have introduced steady periodic solutions of the reaction-diffusion equation (the steady periodic solutions are referred to as the `base state').  We have performed a linear stability analysis of the base state.  In principle, the linear stability analysis relies on the solution of a linearized reaction-diffusion equation, with the base state appearing parametrically therein.  
However, by  formulating the linear-stability analysis in terms of Bloch functions, we have outlined an \textit{a priori} criterion for the base state to be stable, based only on the base-state solution itself, i.e. without needing to solve the linearized equation explicitly.    
%
%
We have compared our approach with an approach obtainable from prior results in the literature, and have demonstrated that our approach outlined herein compares favourably with the prior works.
We anticipate that our results will be of use in the study of reaction-diffusion equations where the base state contains multiple parameters (reaction rate, diffusivity, source amplitude, source lengthscale), since the \textit{a priori} stability criterion will then provide immediate answers as to which parameter values give rise to a stable base state.

\section*{Acknowledgments}

Khang Ee Pang acknowledges funding received from the UCD School of Mathematics and Statistics Summer Research Projects 2018 programme.   

\appendix

\section{The inhomogeneous Allen--Cahn equation with a highly specific source term}
\label{sec:allencahn}

In this section we demonstrate how the functional form $s(x)=-\alpha+\beta\cos(2\kappa x)$ in Equation~\eqref{eq:smatthieu} corresponds to the linearization of the inhomogeneous Allen--Cahn reaction diffusion equation~\cite{allen1972ground} with a specified periodic source term $\sigma(x)$. The approach is slightly contrived and amounts to a `manufactured solution'~\cite{roache2002code}, however, it does demonstrate the applicability of the methods considered in the main part of the paper.   We start with Equations~\eqref{eq:model_C_ss} 
and \eqref{eq:smatthieu} which we assemble here as follows:
\begin{subequations}
\begin{eqnarray}
N'(C_0)&=&s(x)=-\alpha+\beta \cos(2\kappa x),\\
0&=&\sigma(x)+N(C_0)+\frac{\partial^2C_0}{\partial x^2}.\label{eq:appyy}
\end{eqnarray}
\end{subequations}
We take $\partial_x\text{\eqref{eq:appyy}}$ to obtain $\sigma'(x)+N'(C_0)C_0'+C_0'''=0$.  We substitute  for $N'(C_0)=-\alpha+\beta \cos(2\kappa x)$ to obtain $\sigma'+[-\alpha+\beta \cos(2\kappa x)]C_0'+C_0'''=0$.  We propose a solution $C_0=\sin(\kappa x)$; this corresponds to a `manufactured solution' for a highly specific soure term, specifically,
\begin{equation}
\sigma'(x)=[\alpha-\beta \cos(2\kappa x)][\kappa \cos(\kappa x)]+\kappa^3\cos(\kappa x).
\label{eq:appss}
\end{equation}
As such, we have $N'(C_0)=\mathd N/\mathd C_0=-\alpha+\beta \cos (2\kappa x)$.  But $C_0=\sin(\kappa x)$, hence  $\cos (2\kappa x)=1-2C_0^2$, hence
\[
\frac{\mathd N}{\mathd C_0}=-\alpha+\beta(1-2C_0^2).
\]
Integrating once gives
\[
N(C_0)=(\beta-\alpha)C_0-\tfrac{2}{3}\beta C_0^3+\mathrm{Const.}
\]
We set the constant of integration to zero.  We then refer back to the original reaction-diffusion equation~\eqref{eq:model_C} -- this now  corresponds to an Allen--Cahn equation with a periodic source term $\sigma(x)$ given by the integral of~\eqref{eq:appss}.  The specific form of this Allen--Cahn-type equation is as follows:
\begin{equation}
\frac{\partial C}{\partial t}=\sigma(x)+\left[(\beta-\alpha)C-\tfrac{2}{3}\beta C^3\right]+\frac{\partial^2C}{\partial x^2}.
\label{eq:appmodel_C}
\end{equation}



\end{document}